\documentclass[11pt,a4paper]{article}
\usepackage[T1]{fontenc}
\usepackage[utf8]{inputenc}
\usepackage{lmodern}
\usepackage[margin=28mm]{geometry}
\usepackage{amsmath,amssymb,amsthm}
\usepackage[hidelinks]{hyperref}
\usepackage{xcolor}

\usepackage{authblk}

\numberwithin{equation}{section}
\newtheorem{theorem}{Theorem}[section]
\newtheorem{proposition}[theorem]{Proposition}
\newtheorem{corollary}[theorem]{Corollary}

\newcommand{\R}{\mathbb R}
\newcommand{\dd}{\,\mathrm d}
\DeclareMathOperator{\Var}{Var}
\DeclareMathOperator{\tr}{tr}

\newcommand{\II}{\mathrm{II}}
\newcommand{\mc}{\mathrm H}
\newcommand{\V}{\mathcal V}
\newcommand{\sph}{\mathbb S^{n-1}}
\title{A logarithmic Brascamp--Lieb variance inequality}
\author{ Alexander V. Kolesnikov}

\affil{HSE University, Moscow, Russia}

\date{}

\begin{document}
\maketitle

\begin{abstract}
We prove a strengthened Brascamp--Lieb inequality for even
log-concave measures whose positive sublevel bodies satisfy
the local logarithmic Brunn--Minkowski inequality.
The conclusion can be written either by adding a rank-one
term to the Hessian of the potential or by adding an explicit
nonnegative remainder to the variance.
The proof uses sublevel volume variations, a one-dimensional
Brascamp--Lieb inequality, and the coarea formula.
The derivative of the potential under dilations gives an
equality function whenever it is square-integrable.
For potentials homogeneous of degree greater than one, we
obtain a sharp radial remainder.
\end{abstract}

\section{Introduction}
\label{sec:introduction}

Throughout, $n\geq2$, and $o$ denotes the origin of $\R^n$.
Let $V\in C^\infty(\R^n)$ be even, with $D^2V>0$, and assume
that its sublevel sets are compact.
Adding a constant, we normalize $V(o)=0$.
The associated log-concave probability measure is
\[
 Z=\int_{\R^n}e^{-V(x)}\dd x,\qquad
 \dd\mu(x)=Z^{-1}e^{-V(x)}\dd x.
\]
For $f\in L^2(\mu)$, write
\[
 \Var_\mu f=\int_{\R^n}f^2\dd\mu-
                  \left(\int_{\R^n}f\dd\mu\right)^2.
\]
The classical Brascamp--Lieb inequality \cite{BL} states that
\[
 \Var_\mu f\leq
 \int_{\R^n}\langle(D^2V)^{-1}\nabla f,\nabla f\rangle\dd\mu
\]
for smooth functions with finite variance and energy.
Restricting this inequality to even functions leads to the
question of whether the inverse Hessian energy admits a
smaller coefficient or an additional remainder.
For the standard Gaussian measure, the best even coefficient
is $1/2$; see, for example, \cite[inequality (10)]{CKLR}.
Strengthened inequalities for homogeneous potentials are
studied in \cite[Section 8]{CKLR} and
\cite[Section 2]{KLR}.
For radial potentials, Cordero-Erausquin and Rotem
\cite[Theorem 3]{CER} prove an improvement with a scalar
matrix correction. We use that inequality to classify the
equality functions.

There is a close relation between variance inequalities and
second variations of geometric concavity. The direct link between the Brascamb--Lieb inequality and 
the Brunn--Minkowski inequality was established by Bobkov and Ledoux \cite{BL}.
Colesanti \cite{C} derives a boundary Poincar\'e inequality
from the Brunn--Minkowski inequality.
Kolesnikov and Milman \cite{KMPN} develop the corresponding
boundary inequalities and volume variations in the weighted
Riemannian setting.
Their work on local $L^p$-Brunn--Minkowski inequalities
\cite{KM} gives the boundary formulation used below.
Related passages from geometric inequalities to inequalities
for convex functions appear in the Wulff-type inequalities
of Mussnig and Ulivelli \cite{MU}.
Armeniakos and Ulivelli \cite[Theorem 1.1]{AU} obtain an
infinitesimal Pr\'ekopa--Leindler inequality that couples a
boundary Poincar\'e form with a Brascamp--Lieb form.

Recent developments include 
functional forms of the 
log-Brunn--Minkowski inequality, see 
Malliaris, Melbourne, Roberto, and Roysdon
\cite{MMRR},
and the weighted marginal inequalities of Cordero-Erausquin and Eskenazis \cite{CEE}

Our assumption concerns the individual sublevel bodies of $V$.
We require each positive sublevel to satisfy the local
logarithmic Brunn--Minkowski inequality.
This is the infinitesimal form of the logarithmic
Brunn--Minkowski inequality introduced by
B\"or\"oczky, Lutwak, Yang, and Zhang \cite{BLYZ}.
Under this hypothesis, Theorem~\ref{thm:rev-main} improves
the classical inverse Hessian energy by a rank-one correction.
Equivalently, it gives an explicit remainder involving the
derivative of $f$ in the direction $(D^2V)^{-1}\nabla V$.

\noindent\textbf{Acknowledgements.}  
The author thanks Jacopo Ulivelli for discussions, careful reading, and help in preparation of  the manuscript.
The  author  gratefully acknowledges the support of the Basic Research Program of HSE University.

The results of this paper were obtained under assistance of Chat GPT-6.

\subsection{The geometric hypothesis}

By a convex body in $\R^n$ we mean a compact convex set with
nonempty interior, and $|K|$ denotes its Lebesgue measure.
It is origin-symmetric if $K=-K$.
For a smooth convex body, we write
$\dd S=\dd\mathcal H^{n-1}|_{\partial K}$ for surface measure.
Let $K$ be an origin-symmetric convex body with smooth boundary
and positive curvature, and let $\nu$ be its outer unit normal.
We identify the second fundamental form $\II$ with the
positive-definite Weingarten map $D_{\partial K}\nu$, and write
$\mc=\tr\II$ for mean curvature.
On $\partial K$, put
\[
 \varphi(x)=\langle x,\nu(x)\rangle,\qquad x\in\partial K.
\]
In terms of the support function
$h_K(\theta)=\max_{y\in K}\langle y,\theta\rangle$, $\theta\in\sph$,
this means $\varphi(x)=h_K(\nu(x))$. For an even function $u\in C^1(\partial K)$, write
$\nabla_{\partial K}u$ for its tangential gradient.
The local logarithmic Brunn--Minkowski inequality for $K$,
in the boundary formulation of
\cite[Proposition 4.6, with $p=0$]{KM}, is
\begin{equation}\label{eq:rev-local-log-bm}
\begin{aligned}
 &\int_{\partial K}\mc u^2\dd S+
       \int_{\partial K}\frac{u^2}{\varphi}\dd S\\
 &\quad\leq\int_{\partial K}
       \langle\II^{-1}\nabla_{\partial K}u,\nabla_{\partial K}u\rangle\dd S
       +\frac{1}{|K|}
        \left(\int_{\partial K}u\dd S\right)^2.
\end{aligned}
\end{equation}
We require this inequality for every even $u\in C^1(\partial K)$.
Its term involving the integral of $u$ allows arbitrary means.

\subsection{The main theorem}

Our initial space of test functions is
\[
 \mathcal T=
 \{f\in C^\infty(\R^n):
       f\text{ is even and constant outside a compact set}\}.
\]
Define the positive-definite matrix
\[
 \V(x)=D^2V(x)+
       \frac{\nabla V(x)\otimes\nabla V(x)}
            {\langle x,\nabla V(x)\rangle},\qquad x\neq o,
\]
and the two energies
\[
\begin{aligned}
 \mathcal E_0(f)
 &=\int_{\R^n}\langle(D^2V)^{-1}\nabla f,\nabla f\rangle\dd\mu,\\
 \mathcal E_{\V}(f)
 &=\int_{\R^n}\langle\V^{-1}\nabla f,\nabla f\rangle\dd\mu.
\end{aligned}
\]
Strict convexity and evenness give
$\langle x,\nabla V(x)\rangle>0$ for $x\neq o$.
We assign $\V(o)=D^2V(o)$; the value at this single point does
not affect the energy integrals.
We denote by $\mathcal D_0$ and $\mathcal D_{\V}$ the closures of
$\mathcal T$ in the respective form norms
\[
 \bigl(\|f\|_{L^2(\mu)}^2+\mathcal E_0(f)\bigr)^{1/2},
 \qquad
 \bigl(\|f\|_{L^2(\mu)}^2+\mathcal E_{\V}(f)\bigr)^{1/2}.
\]
Since $\V\geq D^2V$, we have
$\mathcal E_{\V}\leq\mathcal E_0$ on $\mathcal T$, and hence
$\mathcal D_0\subseteq\mathcal D_{\V}$.
The forms are closable in the even subspace of $L^2(\mu)$, and
their domains consist of functions with weak gradients.
For functions in these
domains, the gradients and energy integrals below are understood
in this sense.

The form domains allow us to state both versions of the
inequality with all terms finite.

\begin{theorem}[A strengthened even Brascamp--Lieb inequality]
\label{thm:rev-main}
Let $V$ and $\mu$ be as above, and assume that
\eqref{eq:rev-local-log-bm} holds for every sublevel body
$\{V\leq r\}$, $r>0$. Then, for every $f\in\mathcal D_{\V}$,
\begin{equation}\label{eq:rev-matrix}
 \Var_\mu f\leq
 \int_{\R^n}
 \left\langle
 \left(D^2V+\frac{\nabla V\otimes\nabla V}
                    {\langle x,\nabla V\rangle}\right)^{-1}
 \nabla f,\nabla f
 \right\rangle\dd\mu.
\end{equation}
For $f\in\mathcal D_0$, this is equivalently
\begin{equation}\label{eq:rev-remainder}
 \Var_\mu f+\int_{\R^n}
\frac{\langle(D^2V)^{-1}\nabla V,\nabla f\rangle^2}
            {\langle x,\nabla V\rangle+
\langle(D^2V)^{-1}\nabla V,\nabla V\rangle}\dd\mu
 \leq\int_{\R^n}      \langle(D^2V)^{-1}\nabla f,\nabla f\rangle\dd\mu.
\end{equation}
\end{theorem}

Since $\V\geq D^2V$, the energy in
\eqref{eq:rev-matrix} is bounded above by the classical
inverse Hessian energy.
Formula \eqref{eq:rev-remainder} identifies their difference
as a nonnegative integral.
The restriction to $\mathcal D_0$ in this second formulation
ensures that the two energies are finite.

The coefficient one in \eqref{eq:rev-matrix} is optimal
over this class of potentials.
Proposition~\ref{prop:canonical-equality} proves equality for
$f(x)=\langle x,\nabla V(x)\rangle$ under a second-moment
assumption.
For a homogeneous potential, Euler's identity makes the
remainder explicit; see Corollary~\ref{cor:homogeneous-remainder}.

\section{Proof of the main theorem}
\label{sec:proof}

We first express the variance through derivatives of the
logarithm of the sublevel volume.
A one-dimensional Brascamp--Lieb inequality permits a
separate optimization of the level shift at each height.
The local geometric hypothesis estimates the resulting
second variation on each boundary.
Finally, the coarea formula converts these boundary
estimates into an energy inequality on $\R^n$.

\subsection{Volume variations}

We recall the volume variation formulas.
Let $K_\varepsilon$ be a smooth family of convex bodies with
positive curvature, and denote the outer unit normal of
$K_\varepsilon$ by $\nu_\varepsilon$.
For each fixed $x\in\partial K_0$, let
$x_\varepsilon(x)\in\partial K_\varepsilon$ be the point satisfying
\[
 x_0(x)=x,\qquad
 \nu_\varepsilon(x_\varepsilon(x))=\nu_0(x).
\]
A dot denotes differentiation in $\varepsilon$ at zero with the
reference point $x$, and hence the normal direction $\nu_0(x)$,
held fixed. Set
\[
 u(x)=\langle\dot x,\nu_0(x)\rangle,\qquad
 v(x)=\langle\ddot x,\nu_0(x)\rangle.
\]
In the following formulas, $\II$ and $\mc$ refer to the
boundary $\partial K_0$. The first and second variations are
\begin{equation}\label{eq:rev-volume-variations}
\begin{aligned}
 \left.\frac{\dd}{\dd\varepsilon}|K_\varepsilon|
       \right|_{\varepsilon=0}
 &=\int_{\partial K_0}u\dd S,\\
 \left.\frac{\dd^2}{\dd\varepsilon^2}|K_\varepsilon|
       \right|_{\varepsilon=0}
 &=\int_{\partial K_0}
       \left(v+\mc u^2-
       \langle\II^{-1}\nabla_{\partial K_0}u,
                      \nabla_{\partial K_0}u\rangle\right)\dd S.
\end{aligned}
\end{equation}
These formula are establushed in \cite[proof of Theorem 6.6]{KMPN}; see also
\cite{C} for linear perturbations of the support function.

\subsection{Proof of Theorem~\ref{thm:rev-main}}

\begin{proof}
We first consider a function $f\in\mathcal T$ which is constant
near the origin, by approximation arguments one can show that such functions are dense in the original test class for both form
norms.

Subtracting $f(o)$ changes neither the variance nor the gradient.
We may therefore assume that $f$ vanishes in a neighborhood of
the origin.
For sufficiently small $|\varepsilon|$, the function
$V+\varepsilon f$ is strictly convex. Indeed, $D^2f$ vanishes
outside a compact set, and $D^2V$ has a positive lower bound on
that set. Since $V+\varepsilon f$ is even and vanishes at the
origin, its minimum value is zero. This fixes the lower endpoint
of the layer integral.

For $r\geq0$, define the sublevel volume by
\[
 A(r,\varepsilon)=|\{V+\varepsilon f\leq r\}|.
\]
For $r>0$, define its logarithm by
\[
 \ell(r,\varepsilon)=\log A(r,\varepsilon).
\]
Partial derivatives in the two variables are denoted by
$\partial_r$ and $\partial_\varepsilon$. Also, let
\[
 Z_\varepsilon=\int_{\R^n}e^{-V(x)-\varepsilon f(x)}\dd x,
 \qquad Z_0=Z.
\]
The layer-cake formula gives
\begin{equation}\label{eq:rev-layer}
 Z_\varepsilon=\int_0^\infty A(r,\varepsilon)e^{-r}\dd r.
\end{equation}
It follows that the measure $\rho$ on $(0,\infty)$ defined by
\[
 \dd\rho(r)=\frac{A(r,0)e^{-r}}{Z}\dd r
\]
is a probability measure.

Since $A(r,\varepsilon)=\exp(\ell(r,\varepsilon))$,
\[
 \partial_\varepsilon^2 A(r,\varepsilon)
 =A(r,\varepsilon)
   \left(\partial_\varepsilon^2\ell(r,\varepsilon)
         +\bigl(\partial_\varepsilon\ell(r,\varepsilon)\bigr)^2\right).
\]
Differentiating \eqref{eq:rev-layer} twice and using the definition
of $\rho$ therefore yields
\begin{equation}\label{eq:rev-variance}
\begin{aligned}
 \Var_\mu f
 &=\left.\frac{\dd^2}{\dd\varepsilon^2}
                     \log Z_\varepsilon\right|_{\varepsilon=0}\\
 &=\int_0^\infty\partial_\varepsilon^2\ell(r,0)\dd\rho(r)
   +\Var_\rho\bigl(r\mapsto\partial_\varepsilon\ell(r,0)\bigr).
\end{aligned}
\end{equation}

We next estimate the variance on the right-hand side.
Convexity of $V$ and the ordinary Brunn--Minkowski inequality imply
that $r\mapsto A(r,0)^{1/n}$ is concave. 
Consequently,
\[
 \partial_r^2\ell(r,0)
 \leq-\frac1n\bigl(\partial_r\ell(r,0)\bigr)^2 \le 0
\]
(in fact, one can even show that the right-hand side is strictly negative). 
The potential of $\rho$ is $r-\ell(r,0)$, up to an additive
constant. Applying the one-dimensional Brascamp--Lieb inequality
to $r\mapsto\partial_\varepsilon\ell(r,0)$ gives
\[
 \Var_\rho\bigl(r\mapsto\partial_\varepsilon\ell(r,0)\bigr)
 \leq\int_0^\infty
       \frac{\bigl(\partial_r\partial_\varepsilon\ell(r,0)\bigr)^2}
            {-\partial_r^2\ell(r,0)}\dd\rho(r).
\]
Together with
\eqref{eq:rev-variance}, this gives
\begin{equation}\label{eq:rev-one-dimensional}
\begin{aligned}
 \Var_\mu f
 &\leq\int_0^\infty
   \left(\partial_\varepsilon^2\ell(r,0)
   +\frac{\bigl(\partial_r\partial_\varepsilon\ell(r,0)\bigr)^2}
          {-\partial_r^2\ell(r,0)}\right)\dd\rho(r)\\
 &=\int_0^\infty\sup_{c\in\R}
   \left(\partial_\varepsilon^2\ell(r,0)
       +2c\,\partial_r\partial_\varepsilon\ell(r,0)
       +c^2\partial_r^2\ell(r,0)\right)\dd\rho(r).
\end{aligned}
\end{equation}
The supremum is taken separately at each level $r$. The expression
inside it is a concave quadratic polynomial in $c$, whose maximum
is attained at
$c=-\partial_r\partial_\varepsilon\ell(r,0)/\partial_r^2\ell(r,0)$.

Fix $r>0$ and $c\in\R$, and consider the family
\[
 K_\varepsilon=\{V+\varepsilon f\leq r+\varepsilon c\},
 \qquad K_0=\{V\leq r\}.
\]
For small $|\varepsilon|$, these are smooth origin-symmetric
convex bodies with positive curvature. Since
$|K_\varepsilon|=A(r+\varepsilon c,\varepsilon)$, the chain rule gives
\begin{equation}\label{eq:rev-log-volume}
\begin{aligned}
 &\partial_\varepsilon^2\ell(r,0)
       +2c\,\partial_r\partial_\varepsilon\ell(r,0)
       +c^2\partial_r^2\ell(r,0)\\
 &\quad=\left.\frac{\dd^2}{\dd\varepsilon^2}
                           \log|K_\varepsilon|\right|_{\varepsilon=0}\\
 &\quad=\frac{1}{A(r,0)}
       \left.\frac{\dd^2}{\dd\varepsilon^2}
                           |K_\varepsilon|\right|_{\varepsilon=0}
       -\frac{1}{A(r,0)^2}
        \left(\left.\frac{\dd}{\dd\varepsilon}
                           |K_\varepsilon|\right|_{\varepsilon=0}\right)^2.
\end{aligned}
\end{equation}
We will bound this quantity by applying the local logarithmic
Brunn--Minkowski inequality to $K_0$.

Let $\nu$ denote the outer unit normal of $K_0$.
For $x\in\partial K_0$, let $x_\varepsilon$ be the point of
$\partial K_\varepsilon$ with outer normal $\nu(x)$.
This point depends smoothly on $\varepsilon$ and on $x$, 
It satisfies
\[
 (V+\varepsilon f)(x_\varepsilon)=r+\varepsilon c,\qquad
 \nabla(V+\varepsilon f)(x_\varepsilon)
 =|\nabla(V+\varepsilon f)(x_\varepsilon)|\,\nu(x).
\]
Here $x_0=x$.
Write $\dot x=\left.\frac{\dd}{\dd\varepsilon}
x_\varepsilon\right|_{\varepsilon=0}$ and
$\ddot x=\left.\frac{\dd^2}{\dd\varepsilon^2}
x_\varepsilon\right|_{\varepsilon=0}$, and define
\[
 u(x)=\langle\dot x,\nu(x)\rangle,\qquad
 v(x)=\langle\ddot x,\nu(x)\rangle.
\]
These are the normal velocity and normal component of the
acceleration used in \eqref{eq:rev-volume-variations}.

Differentiating the level equation once gives
\[
 \langle\nabla V(x),\dot x\rangle+f(x)=c.
\]
Since $\nabla V(x)=|\nabla V(x)|\nu(x)$, we obtain
\begin{equation}\label{eq:rev-normal-velocity}
 |\nabla V(x)|\,u(x)=c-f(x).
\end{equation}

To determine the full velocity, differentiate the fixed-normal
equation. Its right-hand side has a fixed direction, so
$D^2V(x)\dot x+\nabla f(x)$ is parallel to $\nabla V(x)$.
Thus $\dot x+(D^2V(x))^{-1}\nabla f(x)$ is a scalar multiple
of $(D^2V(x))^{-1}\nabla V(x)$. Taking its scalar product
with $\nabla V(x)$ determines that multiple:
\begin{equation}\label{eq:rev-full-velocity}
\begin{aligned}
 &\dot x+(D^2V(x))^{-1}\nabla f(x)\\
 &\quad=
 \frac{c-f(x)+\langle(D^2V(x))^{-1}\nabla V(x),\nabla f(x)\rangle}
      {\langle(D^2V(x))^{-1}\nabla V(x),\nabla V(x)\rangle}
      (D^2V(x))^{-1}\nabla V(x).
\end{aligned}
\end{equation}
The denominator is positive: $D^2V(x)>0$ and
$\nabla V(x)\neq0$ on the positive level $\{V=r\}$.

Differentiating the level equation a second time gives
\[
 |\nabla V(x)|\,v(x)
 +\langle D^2V(x)\dot x,\dot x\rangle
 +2\langle\nabla f(x),\dot x\rangle=0.
\]
Using the symmetry of $D^2V(x)$, we complete the square:
\begin{align*}
 &\left\langle
 D^2V(x)\bigl(\dot x+(D^2V(x))^{-1}\nabla f(x)\bigr),
 \dot x+(D^2V(x))^{-1}\nabla f(x)\right\rangle\\
 &\quad=\langle D^2V(x)\dot x,\dot x\rangle
       +\left\langle D^2V(x)\dot x,
                         (D^2V(x))^{-1}\nabla f(x)\right\rangle\\
 &\qquad\quad+\langle\nabla f(x),\dot x\rangle
       +\langle\nabla f(x),(D^2V(x))^{-1}\nabla f(x)\rangle\\
 &\quad=\langle D^2V(x)\dot x,\dot x\rangle
       +2\langle\nabla f(x),\dot x\rangle
       +\langle(D^2V(x))^{-1}\nabla f(x),\nabla f(x)\rangle.
\end{align*}
By \eqref{eq:rev-full-velocity}, the expression on the first line equals
\[
 \frac{
 \bigl(c-f(x)+\langle(D^2V(x))^{-1}\nabla V(x),\nabla f(x)\rangle\bigr)^2}
 {\langle(D^2V(x))^{-1}\nabla V(x),\nabla V(x)\rangle}.
\]
Substituting this identity into the second derivative of the level
equation yields
\begin{equation}\label{eq:rev-normal-acceleration}
\begin{aligned}
 |\nabla V(x)|\,v(x)
 &=\langle(D^2V(x))^{-1}\nabla f(x),\nabla f(x)\rangle\\
 &\quad-
 \frac{
 \bigl(c-f(x)+\langle(D^2V(x))^{-1}\nabla V(x),\nabla f(x)\rangle\bigr)^2}
 {\langle(D^2V(x))^{-1}\nabla V(x),\nabla V(x)\rangle}.
\end{aligned}
\end{equation}

The normal velocity $u$ is smooth and even on $\partial K_0$.
Indeed, \eqref{eq:rev-normal-velocity}, evenness of $f$, and oddness
of $\nabla V$ give
\[
 u(-x)=\frac{c-f(-x)}{|\nabla V(-x)|}
      =\frac{c-f(x)}{|\nabla V(x)|}=u(x).
\]
We may therefore apply \eqref{eq:rev-local-log-bm} to $u$ on $K_0$.
Here $\II$ and $\mc$ refer to $\partial K_0$, and
$\varphi(x)=\langle x,\nu(x)\rangle$. The local inequality gives
\[
\begin{aligned}
 &\int_{\partial K_0}
 \left(\mc u^2-\langle\II^{-1}\nabla_{\partial K_0}u,
                            \nabla_{\partial K_0}u\rangle\right)\dd S\\
 &\quad\leq-\int_{\partial K_0}\frac{u^2}{\varphi}\dd S
       +\frac{1}{A(r,0)}
        \left(\int_{\partial K_0}u\dd S\right)^2.
\end{aligned}
\]
Insert this estimate into the second variation formula
\eqref{eq:rev-volume-variations}, and then into
\eqref{eq:rev-log-volume}. The squared first variation cancels,
and we obtain
\begin{equation}\label{eq:rev-geometric-bound}
\begin{aligned}
 &\partial_\varepsilon^2\ell(r,0)
       +2c\,\partial_r\partial_\varepsilon\ell(r,0)
       +c^2\partial_r^2\ell(r,0)\\
 &\quad\leq\frac{1}{A(r,0)}
       \int_{\partial K_0}\left(v-\frac{u^2}{\varphi}\right)\dd S.
\end{aligned}
\end{equation}
On $\partial K_0=\{V=r\}$,
\[
 \varphi(x)=\frac{\langle x,\nabla V(x)\rangle}{|\nabla V(x)|},
 \qquad
 \frac{u(x)^2}{\varphi(x)}
 =\frac{(f(x)-c)^2}
        {\langle x,\nabla V(x)\rangle\,|\nabla V(x)|}.
\]
The denominator $\langle x,\nabla V(x)\rangle$ is positive for
$x\neq o$, by strict convexity and $\nabla V(o)=0$.
Using \eqref{eq:rev-normal-acceleration}, we conclude that
\begin{equation}\label{eq:rev-two-squares}
\begin{aligned}
 &\partial_\varepsilon^2\ell(r,0)
       +2c\,\partial_r\partial_\varepsilon\ell(r,0)
       +c^2\partial_r^2\ell(r,0)\\
 &\quad\leq\frac{1}{A(r,0)}\int_{\{V=r\}}
 \Bigg(
 \langle(D^2V)^{-1}\nabla f,\nabla f\rangle
 -\frac{(f-c)^2}{\langle x,\nabla V\rangle}\\
 &\hspace{48mm}
 -\frac{\bigl(c-f+\langle(D^2V)^{-1}\nabla V,\nabla f\rangle\bigr)^2}
       {\langle(D^2V)^{-1}\nabla V,\nabla V\rangle}
 \Bigg)\frac{\dd S}{|\nabla V|}.
\end{aligned}
\end{equation}
All functions in this boundary integral are evaluated at its
integration point $x$.

We now combine the two squared terms. For real numbers $s,t$ and
positive numbers $\alpha,\beta$, the weighted Cauchy--Schwarz
inequality states that
\[
 \frac{s^2}{\alpha}+\frac{t^2}{\beta}
 \geq\frac{(s+t)^2}{\alpha+\beta}.
\]
Indeed, apply ordinary Cauchy--Schwarz to
$(s/\sqrt\alpha,t/\sqrt\beta)$ and $(\sqrt\alpha,\sqrt\beta)$.
In \eqref{eq:rev-two-squares}, the two numerators before squaring
sum to $\langle(D^2V)^{-1}\nabla V,\nabla f\rangle$.
The inequality therefore gives, pointwise,
\begin{equation}\label{eq:rev-weighted-cs}
\begin{aligned}
 \frac{(f-c)^2}{\langle x,\nabla V\rangle}
 +\frac{\bigl(c-f+\langle(D^2V)^{-1}\nabla V,\nabla f\rangle\bigr)^2}
    {\langle(D^2V)^{-1}\nabla V,\nabla V\rangle}
 \geq
\frac{\langle(D^2V)^{-1}\nabla V,\nabla f\rangle^2}
      {\langle x,\nabla V\rangle+
    \langle(D^2V)^{-1}\nabla V,\nabla V\rangle}.
\end{aligned}
\end{equation}
For another use of this weighted estimate in a boundary
calculation, see \cite[Lemma 8.1]{KM}.

After this substitution, the right-hand side of
\eqref{eq:rev-two-squares} is independent of $c$.
Taking the supremum over $c$ in \eqref{eq:rev-one-dimensional}
and substituting $\dd\rho(r)=Z^{-1}A(r,0)e^{-r}\dd r$ gives
\[
\begin{aligned}
 \Var_\mu f
 &\leq\frac1Z\int_0^\infty e^{-r}\int_{\{V=r\}}
 \Bigg(
 \langle(D^2V)^{-1}\nabla f,\nabla f\rangle
  -\frac{\langle(D^2V)^{-1}\nabla V,\nabla f\rangle^2}
       {\langle x,\nabla V\rangle+
        \langle(D^2V)^{-1}\nabla V,\nabla V\rangle}
 \Bigg)\frac{\dd S}{|\nabla V|}\dd r.
\end{aligned}
\]
The expression in parentheses is nonnegative. This follows from
Cauchy--Schwarz in the inverse Hessian metric:
\[
 \langle(D^2V)^{-1}\nabla V,\nabla f\rangle^2
 \leq
 \langle(D^2V)^{-1}\nabla V,\nabla V\rangle
 \langle(D^2V)^{-1}\nabla f,\nabla f\rangle.
\]
It vanishes near the origin and outside a compact set, since
$f$ is constant on both regions. The coarea formula thus yields
\[
 \Var_\mu f\leq\int_{\R^n}
 \left(
 \langle(D^2V)^{-1}\nabla f,\nabla f\rangle
 -\frac{\langle(D^2V)^{-1}\nabla V,\nabla f\rangle^2}
       {\langle x,\nabla V\rangle+
        \langle(D^2V)^{-1}\nabla V,\nabla V\rangle}
 \right)\dd\mu.
\]
This proves \eqref{eq:rev-remainder} for the test functions
currently under consideration.

The rank-one inverse formula gives, for $x\neq o$,
\[
 \V^{-1}
 =(D^2V)^{-1}
 -\frac{\bigl((D^2V)^{-1}\nabla V\bigr)
             \otimes\bigl((D^2V)^{-1}\nabla V\bigr)}
       {\langle x,\nabla V\rangle+
        \langle(D^2V)^{-1}\nabla V,\nabla V\rangle}.
\]
Hence the preceding integrand is
$\langle\V^{-1}\nabla f,\nabla f\rangle$, and
\eqref{eq:rev-matrix} follows for test functions constant near
the origin.

We approximate  any $f\in\mathcal T$, 
by functions constant near the origin 
in the matrix form norm.
If $f\in\mathcal D_0$, then $f\in\mathcal D_{\V}$ and the
inverse formula applies almost everywhere to its weak gradient.
The inverse Hessian Cauchy--Schwarz estimate above shows that
the remainder in \eqref{eq:rev-remainder} is bounded by
$\mathcal E_0(f)$, so all terms are finite.
The inverse formula identifies this remainder with
$\mathcal E_0(f)-\mathcal E_{\V}(f)$ and proves that
\eqref{eq:rev-matrix} and \eqref{eq:rev-remainder} are equivalent
on $\mathcal D_0$.
\end{proof}

\subsection{Equality and sharpness}
\label{sec:sharpness}

The function $\langle x,\nabla V(x)\rangle$ is the derivative
at $t=0$ of $V(e^t x)$.
The following proposition shows that it attains equality
whenever it is square-integrable.

\begin{proposition}[Equality under dilations]
\label{prop:canonical-equality}
Assume the hypotheses of Theorem~\ref{thm:rev-main}, and suppose
that the function
\[
 f(x)=\langle x,\nabla V(x)\rangle
\]
belongs to $L^2(\mu)$.
Then $f\in\mathcal D_{\V}$ and equality holds in
\eqref{eq:rev-matrix}. The same is true for every function
$c_0+c_1f$, where $c_0,c_1\in\R$.
In particular, the coefficient one multiplying the right-hand
side of \eqref{eq:rev-matrix} is optimal.
\end{proposition}

\begin{proof}
For $x\neq o$, direct differentiation gives
\[
 \V(x)x
 =D^2V(x)x+\nabla V(x)
 =\nabla f(x).
\]
Thus the matrix energy density of $f$ is
$\langle x,\nabla f(x)\rangle$.
One has
\[
 \int_{\R^n}f\dd\mu=n,\qquad
 \int_{\R^n}\langle x,\nabla f\rangle\dd\mu
 =\int_{\R^n}f(f-n)\dd\mu.
\]
Consequently,
\begin{equation}\label{eq:canonical-equality}
 \mathcal E_{\V}(f)
 =\int_{\R^n}f^2\dd\mu-n^2
 =\Var_\mu f.
\end{equation}
The function $f$ is nonconstant, since $f(o)=0$ and
$f(x)>0$ for $x\neq o$. Hence its variance is positive.
This proves optimality of the coefficient, and the assertion
for $c_0+c_1f$ follows by scaling and translation.
\end{proof}

The finite-moment assumption ensures that this equality function
belongs to the stated form domain. The computation does not
require finite classical inverse Hessian energy.
If that classical energy is finite, the same cutoff argument
places $f$ in $\mathcal D_0$, and equality holds in
\eqref{eq:rev-remainder} as well.


The coefficient of the matrix correction is also optimal
uniformly over the potentials in the theorem. For the standard
Gaussian potential $V(x)=|x|^2/2$, the test function $f=V$
has variance $n/2$. Write $I$ for the identity matrix.
If the matrix correction were multiplied
by a number $\tau\geq0$, its energy would be
\[
 \int_{\R^n}
 \left\langle
 \left(I+\tau\frac{x\otimes x}{|x|^2}\right)^{-1}x,x
 \right\rangle\dd\mu
 =\frac{n}{1+\tau}.
\]
The inequality can therefore hold only when $\tau\leq1$.

\section{Additional remarks}

\subsection{Homogeneous potentials}

For homogeneous potentials, the geometry of the positive
sublevels is determined by a single convex body.
The remainder inequality is equivalent to the local
logarithmic Brunn--Minkowski inequality for that body.

Let $K\subset\R^n$ be an origin-symmetric convex body with
smooth boundary and positive curvature. Its Minkowski functional
is the norm
\[
 \|x\|_K=\inf\{t>0:x\in tK\}.
\]
Fix $\alpha>1$, and consider
\[
 V(x)=\frac{\|x\|_K^\alpha}{\alpha},\qquad
 \dd\mu(x)=Z^{-1}e^{-V(x)}\dd x,\qquad
 Z=\int_{\R^n}e^{-V(x)}\dd x.
\]
This potential is continuous at the origin and smooth with
positive-definite Hessian away from it.
The coefficient matrices in the energy integrals are defined
for $x\neq o$ and may be assigned arbitrary positive-definite
values at the origin.
It need not satisfy the smoothness assumptions of
Theorem~\ref{thm:rev-main} at the origin.
Throughout this section, $\mathcal E_0$, $\mathcal E_{\V}$,
$\mathcal D_0$, and $\mathcal D_{\V}$ are defined by the same
formulas as before, using this potential and this test class.
The gradients of form limits are understood weakly away from
the origin; the value at that single point does not affect an
energy integral.

\begin{corollary}[The homogeneous radial remainder]
\label{cor:homogeneous-remainder}
\label{cor:homogeneous-equivalence}
For the fixed degree $\alpha>1$, $K$ satisfies
\eqref{eq:rev-local-log-bm} if and only if every
$f\in\mathcal D_0$ satisfies
\begin{equation}\label{eq:homogeneous-remainder}
 \Var_\mu f+
 \int_{\R^n}\frac{\langle x,\nabla f\rangle^2}
                    {\alpha^2(\alpha-1)V}\dd\mu
 \leq\mathcal E_0(f).
\end{equation}
Under these equivalent conditions, \eqref{eq:rev-matrix}
also holds for every $f\in\mathcal D_{\V}$.
The function $f=V$ belongs to $\mathcal D_0$ and gives
equality in \eqref{eq:homogeneous-remainder}.
Consequently, the coefficient one of the remainder is optimal.
\end{corollary}

The result was proved in 
\cite[Theorem 2.4]{KLR}.
Here the forward implication follows from Theorem~\ref{thm:rev-main} and Euler's
identity.

\subsection{A limitation of the Euclidean Bochner argument}

Let
\[
 L_Eu=\Delta u-\langle\nabla V,\nabla u\rangle,
\]
and write $\|A\|_{\mathrm{HS}}^2=\sum_{i,j}A_{ij}^2$ for
the squared Hilbert--Schmidt norm of a matrix.
For smooth compactly supported $u$, the integrated Euclidean
Bochner identity is
\begin{equation}\label{eq:euclidean-bochner}
 \int_{\R^n}(L_Eu)^2\dd\mu
 =\int_{\R^n}\|D^2u\|_{\mathrm{HS}}^2\dd\mu
       +\int_{\R^n}\langle D^2V\nabla u,\nabla u\rangle\dd\mu.
\end{equation}
See, for example, \cite[Section 4.3, formula (43)]{KLR}.
To obtain the matrix in \eqref{eq:rev-matrix} through the
usual duality argument, one might try to supply its additional
term by the estimate
\begin{equation}\label{eq:false-bochner-estimate}
 \int_{\R^n}\|D^2u\|_{\mathrm{HS}}^2\dd\mu
 \stackrel{?}{\geq}
 \int_{\R^n}\frac{\langle\nabla u,\nabla V\rangle^2}
                    {\langle x,\nabla V\rangle}\dd\mu
 \qquad(u\text{ even}).
\end{equation}
This estimate fails even for a Gaussian potential.
In dimension two, take
\[
 V(x_1,x_2)=\frac{9x_1^2+x_2^2}{2},
 \qquad u(x_1,x_2)=x_1x_2.
\]
The function $u$ is even, and
$\|D^2u\|_{\mathrm{HS}}^2=2$.
With $y_1=3x_1$ and $y_2=x_2$, the second integral in
\eqref{eq:false-bochner-estimate} equals
\[
 \frac{100}{9}\int_{\R^2}
       \frac{y_1^2y_2^2}{y_1^2+y_2^2}
       e^{-|y|^2/2}\frac{\dd y}{2\pi}
 =\frac{100}{9}\cdot2\cdot\frac18
 =\frac{25}{9}>2.
\]
Here the radial second moment of the standard Gaussian in
$\R^2$ is two, and the angular mean of
$\cos^2\theta\sin^2\theta$ is $1/8$.

The same failure occurs for smooth compactly supported even
functions. Indeed, multiply $u$ by a smooth radial cutoff
$\chi_R$ equal to one on the ball of radius $R$ and zero
outside the ball of radius $2R$, with
$|\nabla\chi_R|\leq C/R$ and $|D^2\chi_R|\leq C/R^2$.
The Hessians of $\chi_Ru$ converge to $D^2u$ in $L^2(\mu)$.
Their gradients also converge in the quadratic integral
on the right of \eqref{eq:false-bochner-estimate}, since
Gaussian decay controls the polynomial cutoff errors.
The two integrals therefore tend to $2$ and $25/9$.
Thus this particular strengthening of the Euclidean
Bochner identity cannot give \eqref{eq:rev-matrix}.

\subsection{$p$-Brunn--Minkowski inequality with $p<0$}

Our result  does not cover, for instance, the result of Cordero-Erausquin and Rotem on Breascamb--Lieb type inequalities for radially symmetric log-concave measures.
The reason for this: the Euclidean balls do satisfy 
the strongest $p$-BM inequality, namely with $p=-n$ and we use in our proof 
only the weaker log Brunn--Minkowsmi inequality. However, it is possible to extend our result to potentials with level sets satisfying $p$-BM with negative $p$.


\begin{thebibliography}{99}

\bibitem{AU}
S. Armeniakos and J. Ulivelli,
\emph{A new infinitesimal form of the Pr\'ekopa--Leindler
inequality with multiplicative structure and applications},
preprint, \href{https://arxiv.org/abs/2602.10812v2}
{arXiv:2602.10812v2} (2026).

\bibitem{BL}
 S. G. Bobkov and M. Ledoux, \emph {From Brunn–Minkowski to Brascamp–Lieb and to logarithmic Sobolev inequalities}, Geom. Funct. Anal. 10 (2000), no. 5, 1028–1052. 



\bibitem{BL}
H.~J. Brascamp and E.~H. Lieb,
\emph{On extensions of the Brunn--Minkowski and
Pr\'ekopa--Leindler theorems, including inequalities for log concave
functions, and with an application to the diffusion equation},
J. Funct. Anal. \textbf{22} (1976), no.~4, 366--389.
\href{https://doi.org/10.1016/0022-1236(76)90004-5}
{doi:10.1016/0022-1236(76)90004-5}.

\bibitem{BLYZ}
K.~J. B\"or\"oczky, E. Lutwak, D. Yang, and G. Zhang,
\emph{The log-Brunn--Minkowski inequality},
Adv. Math. \textbf{231} (2012), nos.~3--4, 1974--1997.
\href{https://doi.org/10.1016/j.aim.2012.07.015}
{doi:10.1016/j.aim.2012.07.015}.

\bibitem{C}
A. Colesanti,
\emph{From the Brunn--Minkowski inequality to a class of
Poincar\'e-type inequalities},
Commun. Contemp. Math. \textbf{10} (2008), no.~5, 765--772.
\href{https://doi.org/10.1142/S0219199708002971}
{doi:10.1142/S0219199708002971}.

\bibitem{CKLR}
A. Colesanti, A.~V. Kolesnikov, G. Livshyts, and L. Rotem,
\emph{On weighted Blaschke--Santal\'o and strong Brascamp--Lieb
inequalities},
preprint, \href{https://arxiv.org/abs/2409.11503v2}
{arXiv:2409.11503v2} (2026).


\bibitem{CEE}
D. Cordero-Erausquin and A. Eskenazis, \emph{Concavity principles for weighted marginals}, preprint, arXiv:2506.16941 (2025).


\bibitem{CER}
D. Cordero-Erausquin and L. Rotem,
\emph{Improved log-concavity for rotationally invariant measures
of symmetric convex sets},
Ann. Probab. \textbf{51} (2023), no.~3, 987--1003.
\href{https://doi.org/10.1214/22-AOP1604}
{doi:10.1214/22-AOP1604}.




\bibitem{KLR}
A.~V. Kolesnikov, G. Livshyts, and L. Rotem,
\emph{On $p$-Brunn--Minkowski and Brascamp--Lieb inequalities},
preprint, \href{https://arxiv.org/abs/2507.12099v3}
{arXiv:2507.12099v3} (2026).

\bibitem{KM}
A.~V. Kolesnikov and E. Milman,
\emph{Local $L^p$-Brunn--Minkowski inequalities for $p<1$},
Mem. Amer. Math. Soc. \textbf{277} (2022), no.~1360, v+78 pp.
\href{https://doi.org/10.1090/memo/1360}
{doi:10.1090/memo/1360}.

\bibitem{KMPN}
A.~V. Kolesnikov and E. Milman,
\emph{Poincar\'e and Brunn--Minkowski inequalities on the boundary
of weighted Riemannian manifolds},
Amer. J. Math. \textbf{140} (2018), no.~5, 1147--1185.
\href{https://doi.org/10.1353/ajm.2018.0027}
{doi:10.1353/ajm.2018.0027}.


\bibitem{MMRR}
A. Malliaris, J. Melbourne, C. Roberto, and M. Roysdon, \emph{Functional liftings of restricted geometric inequalities}, preprint, arXiv:2508.15247v3 (2025)


\bibitem{MU}
F. Mussnig and J. Ulivelli,
\emph{Inequalities and counterexamples for functional intrinsic
volumes and beyond},
J. Lond. Math. Soc. (2) \textbf{113} (2026), no.~1,
Art.~e70422.
\href{https://doi.org/10.1112/jlms.70422}
{doi:10.1112/jlms.70422}.

\end{thebibliography}
\end{document}